\documentclass{article}
\usepackage{amsmath,  amsthm, amssymb}
\numberwithin{equation}{section}

\newtheorem{lemma}{Lemma}

\begin{document}

\author{Ajai Choudhry}
\title{Three  integers whose sum, product and \\
the sum of the products of the integers,\\
 taken two at a time, are perfect squares}

\date{}
\maketitle
\abstract{Euler had considered the problem of finding three integers whose sum,  product, and also the sum of the products of the integers, taken two at a time, are all perfect squares. Euler's methods of solving the problem lead to parametric solutions in terms of polynomials of high degrees and his   numerical solutions consisted of very large integers. We obtain, by a new method,  several parametric  solutions  given by polynomials of much smaller degrees and thus we get  a number of numerically small solutions of the problem.} 
\medskip

\noindent Keywords: three integers whose elementary symmetric functions are all squares.
\medskip

\bigskip
 
\section{Introduction}\label{intro}
This paper is concerned with the problem of finding three integers $a, b$ and $c$ whose sum,  product, and also the sum of the products of the integers, taken two at a time, are all perfect squares. 

The  problem was first considered by Euler in 1760. While Euler's original paper is in Latin \cite{LE1}, an English translation of the paper has recently been published in 2022 \cite{SW}. Euler has mentioned that he spent ``a long time in vain" searching for such numbers and then, ``almost unexpectedly" found a solution by ``a remarkable method". He noted that any solution of the problem in rational numbers readily yields, on multiplying the rational numbers by a suitable perfect square,  a solution of the problem in integers.  His method of solving leads to very large numbers. In fact, his method yields parametric solutions in terms of a single rational parameter.  While Euler did not carry out the cumbersome computations to find the parametric solutions, we found, using the software MAPLE, that in  the simplest parametric solution, obtained by Euler's method, the values of $a, b, c $ are  given by polynomials of degree 132. The smallest numerical solution found by Euler, consisting of integers with 13 digits, is as follows:
\[
1633780814400,  \quad  252782198228,\quad  3474741058973.
\]

Euler revisited the problem in 1779 \cite{LE2}. While he again referred to the problem as being ``very difficult" and giving rise to ``immense numbers", he obtained relatively       simpler solutions in smaller integers by imposing the additional condition  that all the three integers  $a, b $ and $c$ must be  perfect squares. Euler's new method of solving also leads to parametric solutions. Since Euler did not find the parametric solutions explicitly, we did the computations on MAPLE and found several parametric  solutions in terms of univariate polynomials. The simplest such solution may be written as follows:
\begin{equation*}
\begin{aligned}
a & = (t - 1)^2(t + 1)^2(t^2 + 5)^2(t^4 - 10t^2 + 5)^2, \\
b & = 256t^4(t - 1)^2(t + 1)^2(t^2 - 3)^2, \\
 c& = 4t^2(t^2 - 3)^2(t^4 - 10t^2 + 5)^2,
\end{aligned}
\end{equation*}
where $t$ is an arbitrary rational parameter. The smallest numerical solution found by Euler is the triad of integers: 81, 784 and 186624.

In 1899,  Fauquembergue \cite{EF} obtained the simpler solution $(a, b, c)= 4s^4t^2(s^2 + t^2), 4t^4(s^2 + t^2)s^2, (s^4 - t^4)^2$      which, on taking $(s, t)=(2, 1)$, yields  the numerically small solution $(a, b, c)= (320, 80, 225)$.   Solutions that are numerically smaller, as compared to Euler's solutions, have also been found recently using ``brute force methods" \cite{SW}.

In this paper we show that each of the three integers $a, b, c$ must be a sum of two squares of rational numbers of which one may possibly be   zero. Thus, in any solution, either exactly one of the integers is a square, or all three of them are squares, or none of them is a square.  We find parametric solutions of all three types, and we show how infinitely many parametric solutions can be obtained. 

Specifically, we obtain four   parametric solutions  in which one of the three integers $a, b, c$ is a square --- two of these solutions are given by simple polynomials of degree 8 while the remaining two are given by polynomials of degree 20. These four solutions,  on suitable specialization, yield four parametric solutions of degrees 12, 16, 36 and 40, respectively, in which all the three integers $a, b, c$ are  squares. We also obtain a parametric solution given by polynomials of degree 32 in which none of the three integers is a square. 

While one of our  solutions is the same as that found by Fauquembergue, all other solutions are new and are given by polynomials of much lower degrees as compared to the solutions generated by Euler's methods. Further, 
the  parametric solutions obtained in the paper yield  several numerically small solutions of the problem, the following four triads being examples:
\[
\{45, 64, 180\}, \quad \{81, 160, 1440\}, \quad \{35^2, 72^2, 96^2\}, \quad \{72, 136, 153\}. 
\]

\section{Three integers whose sum, product and the sum of the products, taken two at a time,
 are perfect squares}\label{threenumbers}  
We first prove a preliminary lemma concerning rational numbers that are expressible as a ratio of two rational numbers that are sums of two squares. We then apply the lemma  to show that the three integers that we are seeking must all be sums of two squares of rational numbers.

\subsection{A preliminary lemma}\label{prel}
\begin{lemma}\label{prelemma}: If two nonzero rational numbers $\alpha$ and $\beta$ are both expressible as sums of two squares of rational numbers, then their ratio $\alpha/\beta$ is also expressible as a sum of two squares of rational numbers.
\end{lemma}
\begin{proof} Let $\alpha= \alpha_1^2+ \alpha_2^2$ and $\beta= \beta_1^2+ \beta_2^2$ where $\alpha_i, \beta_i, i=1, 2$ are all rational numbers.  Since $\alpha$ and $\beta$ are both nonzero, we may  define  two rational numbers $m$ and $n$ as follows:
\[
m = (\alpha_1\beta_1 + \alpha_2\beta_2)/(\beta_1^2 + \beta_2^2), \quad n= (\alpha_1\beta_2 - \alpha_2\beta_1)/(\beta_1^2 + \beta_2^2).
\]
It is readily verified that $m^2+n^2=(\alpha_1^2 + \alpha_2^2)/(\beta_1^2 + \beta_2^2) = \alpha/\beta$. This proves the lemma. 
\end{proof}

We note that when $\alpha_1/\alpha_2=\beta_1/\beta_2$, then $n=0$ and $\alpha/\beta$ simply becomes $m^2$. 

\subsection{Parametric solutions in which one of the three integers is a square}\label{onesquare}

We now consider the problem of finding three rational numbers $a, b$ and $c$ that satisfy the conditions of our problem, that is, there must exist rational numbers $f, g $ and $h$ such that
\begin{equation}
\begin{aligned}
a+b +c & = f^2, \\
ab+bc+ca & =g^2, \\
abc & =h^2.
\end{aligned}
\label{Eulerprob}
\end{equation}

Thus the three elementary symmetric functions of $a, b, c$ are all squares, and   $a, b$ and $c$ are the roots of the following cubic equation:
\begin{equation}
x^3-f^2x^2+g^2x-h^2=0. \label{cubeqn}
\end{equation}
As already noted by Euler, any solution of our problem in rational numbers yields, on appropriate scaling, a solution in integers. 

Eq.        \eqref{cubeqn} may be written as
\[
x=(f^2x^2+h^2)/(x^2+g^2).
\]
Thus any rational root of Eq.        \eqref{cubeqn} is expressible as a ratio of two rational numbers both of which are sums of squares of two rational numbers. It now follows from Lemma \ref{prelemma}  that every rational root of Eq.        \eqref{cubeqn} is expressible as a sum of two squares of  rational numbers. Thus, we may take $x=p^2+q^2$ as one of the roots of the cubic Eq.        \eqref{cubeqn} where $p$ and $q$ are rational numbers.

 On substituting $x=p^2+q^2$ in Eq.        \eqref{cubeqn} and using the identity 
\[
(\alpha_1^2+ \alpha_2^2)(\beta_1^2 + \beta_2^2)=(\alpha_1\beta_1 + \alpha_2\beta_2)^2+(\alpha_1\beta_2 - \alpha_2\beta_1)^2,
\]
we may write $x^3+g^2x=x(x^2+g^2)$ as a sum of two squares, and accordingly, Eq.        \eqref{cubeqn} may be written as
\begin{equation}
(p^3 + pq^2 + qg)^2 + (p^2q + q^3 - pg)^2 -f^2(p^2+q^2)^2 -h^2=0. \label{cubeqn1}
\end{equation}
After suitable transposition of terms and factorization, Eq.        \eqref{cubeqn1} may be written as follows:
\begin{multline}
(p^3 - p^2f + pq^2 - q^2f + qg)(p^3 + p^2f + pq^2 + q^2f + qg)\\
 =(h-p^2q - q^3 + pg)(h+p^2q + q^3 - pg).\label{cubeqn2}
\end{multline}

Assuming that both sides of Eq.        \eqref{cubeqn2} are $\neq 0$, there exists a nonzero rational number $m$ such that
\begin{equation}
p^3 - p^2f + pq^2 - q^2f + qg=m(h-p^2q - q^3 + pg), \label{cubeqn2a}
\end{equation}
and now Eq.        \eqref{cubeqn2} reduces to
\begin{equation}
m(p^3 + p^2f + pq^2 + q^2f + qg)=h+p^2q + q^3 - pg. \label{cubeqn2b}
\end{equation}

Eqs. \eqref{cubeqn2a} and \eqref{cubeqn2b} may be considered as two linear equations in the variables $f$ and $g$, and on solving them, we get
\begin{equation}
\begin{aligned}
f & = -((p^4 + 2p^2q^2 + q^4 - qh)m^2 - 2mph + p^4 + 2p^2q^2 + q^4 + qh)\\
& \quad \quad \times ((m^2p - 2mq - p)(p^2 + q^2))^{-1}, \\
g &  = ((p^2q + q^3 - h)m^2 + (2p^3 + 2pq^2)m - p^2q - q^3 - h)/(m^2p - 2mq - p).
\end{aligned}
\label{valfg}
\end{equation}

With the values of $f$ and $g$ given by \eqref{valfg}, Eq.        \eqref{cubeqn} has a root $x=p^2+q^2$, and indeed, Eq.        \eqref{cubeqn} now reduces to 
\begin{multline}
(x-p^2-q^2)\{(m^2p - 2mq - p)^2(p^2 + q^2)^2x^2  -((m^2q + 2pm - q)^2h^2\\
 - 2(p^2 + q^2)^2(m^2 + 1)(m^2q + 2pm - q)h + (p^2 + q^2)^3(m^2q + 2pm - q)^2)x\\
 + h^2(p^2 + q^2)(m^2p - 2mq - p)^2\} =0. \label{cubeqnred1}
\end{multline}

We remove the first factor $x-p^2-q^2$, and write
\begin{equation}
\begin{aligned}
p & = (m^2t + 2ms - t)/(m^2 + 1), \\
q & = (m^2s - 2mt - s)/(m^2 + 1), \\
 h & = u(s^2 + t^2),
\end{aligned}
\label{valpqh}
\end{equation}
where $s, t$ and $u$ are rational parameters, and now  Eq.        \eqref{cubeqnred1} reduces, after removing the factors $(s^2 + t^2)^2(m^2 + 1)^2$, to the following quadratic equation in $x$:
\begin{equation}
t^2x^2 - s(s^3 - 2s^2u + st^2 + su^2 - 2t^2u)x + u^2t^2(s^2 + t^2)=0. \label{quadxu}
\end{equation}

The discriminant of Eq.~\eqref{quadxu} with respect to $x$ is $s^4\phi(s,t,u)$, where
\begin{equation}
\begin{aligned}
\phi(s,t,u)&  =u^4 - 4(s^2 + t^2)u^3/s + 2(s^2 + t^2)(3s^4 + 2s^2t^2 - 2t^4)u^2/s^4 \\
& \quad \quad - 4(s^2 + t^2)^2u/s + (s^2 + t^2)^2.
\end{aligned}
\label{disx}
\end{equation}

 Eq.        \eqref{quadxu} will have two rational roots if its discriminant, and hence also $\phi(s,t,u)$, is a perfect square, that is, there must exist a rational number $v$ such that the following equation has a solution in rational numbers:
\begin{equation}
v^2 =u^4+a_1u^3+a_2u^2+a_3u+a_4,
\label{quarticec}
\end{equation}
where 
\begin{equation}
\begin{aligned}
a_1  =  &  - 4(s^2 + t^2)/s, \quad & a_2  =  &   2(s^2 + t^2)(3s^4 + 2s^2t^2 - 2t^4)/s^4, \\
 a_3  =  &  - 4(s^2 + t^2)^2/s, \quad & a_4  =  &    (s^2 + t^2)^2.
\end{aligned}
\label{valaphi}
\end{equation}

Each solution of Eq.        \eqref{quarticec} yields a value of $u$ which makes the discriminant of Eq.        \eqref{quadxu} a perfect square, and hence Eq.        \eqref{quarticec} has two rational roots. Thus, Eq.        \eqref{cubeqn} has three rational roots and we  obtain the desired numbers $a, b $ and $c$.
 
The right-hand side of Eq.        \eqref{quarticec} is  a quartic function of $u$ in which   the coefficients of both $u^4$ and $u^0$ are perfect squares. Fermat had described a method (as quoted by Dickson \cite[p.\ 639]{Di}) of making such a quartic function a perfect square, and using Fermat's method, we readily find the following two solutions of Eq.        \eqref{quarticec}:
\begin{align}
(u, v) & = (2s^3/(s^2 - t^2), -(s^6 - s^4t^2 - 5s^2t^4 + t^6)/(s^2 - t^2)^2), \label{Fermatsol1}\\
(u, v) & = ((s^4 - t^4)/(2s^3), (s^2 + t^2)(s^6 - s^4t^2 - 5s^2t^4 + t^6)/(4s^6)). \label{Fermatsol2}
\end{align}

The value of $u$ given by each of the two solutions \eqref{Fermatsol1} and \eqref{Fermatsol2} makes the discriminant of Eq.        \eqref{quadxu} a perfect square, and hence this equation has two rational roots, and hence also Eq.        \eqref{cubeqn} has three rational roots. Using the relations \eqref{valpqh}, these three roots may be expressed as  rational functions  of $s$ and $t$, and on appropriate scaling, we get the values of $a, b, c$ in terms of polynomials in $s$ and $t$. 

Specifically, we note that, on taking $u= 2s^3/(s^2 - t^2)$, as in the solution \eqref{Fermatsol1}, Eq.        \eqref{quadxu} has the  two rational roots,
\begin{equation}
x=s^2(s^2 + t^2)/t^2 \quad {\rm and} \quad x = 4s^4t^2/(s^2 - t^2)^2, \label{rootpairquadxu}
\end{equation}
and we thus obtain, on appropriate scaling,  the following triad $a, b, c$ of integers, in terms of arbitrary integer parameters $s$ and $t$, such that the elementary symmetric functions of $a, b$ and $c$ are all perfect squares:
\begin{equation}
(a, b, c)  = (t^2(s^2 - t^2)^2(s^2+t^2), s^2(s^2 - t^2)^2(s^2+t^2), 4s^4t^4). \label{parmsol1}
\end{equation}

Similarly,  the value of $u$ given by the solution \eqref{Fermatsol2} of Eq.        \eqref{quarticec} yields two rational solutions of Eq.        \eqref{quadxu} and, as before, we obtain, on appropriate scaling, the following triad $a, b, c$:
\begin{equation}
(a, b, c)  = (4s^4t^2(s^2 + t^2), 4s^2t^4(s^2 + t^2), (s^4 - t^4)^2). \label{parmsol2}
\end{equation}
The solution \eqref{parmsol2} is the same as Fauquembergue's solution mentioned in the Introduction.

According to a theorem of Choudhry \cite[Theorem 4.1, pp. 789-790]{Ch}, if   $(u_1,\,v_1)$ and $(u_2,\,v_2)$ are two rational solutions of Eq.        \eqref{quarticec} with $u_1 \neq u_2$, a new rational solution of  Eq.        \eqref{quarticec} is given by $(u_{12},\,v_{12})$ where
\begin{multline} 
u_{12}=\{-2v_1v_2+2(u_1-u_2)(u_2v_1-u_1v_2)+a_1(u_1+u_2)u_1u_2\\
 +2a_2u_1u_2+a_3(u_1+u_2)+2a_4+2(u_1^2-u_1u_2+u_2^2)u_1u_2\}\\
 \times \{(u_1-u_2)(2v_1-2v_2+a_1(u_1-u_2)+2u_1^2-2u_2^2)\}^{-1}, \label{valu12}
\end{multline}
provided the denominator on the right-hand side of \eqref{valu12} $ \neq 0$. We omit the value of $v_{12}$ given in  \cite[Theorem 4.1, pp. 789-790]{Ch} as it is cumbersome to write and is not needed for our computations. 

Applying Choudhry's theorem to Eq.        \eqref{quarticec}, taking $(u_1, v_1)$ as the obvious solution $(0, s^2+t^2)$ and $(u_2, v_2)$ as the solution \eqref{Fermatsol1}, we find a solution of Eq.        \eqref{quarticec} in which 
\begin{equation}
u =   (s^4-t^4) (3 s^6 + s^4 t^2 + s^2 t^4 - t^6)/(2 s^3 (s^6 - s^4 t^2 - 5 s^2 t^4 + t^6)), \label{RMJsol1}
\end{equation}
and this yields the following solution of our problem:
\begin{equation}
\begin{aligned}
a & = 4s^4t^2(s^2 + t^2)(s^6 - s^4t^2 - 5s^2t^4 + t^6)^2,\\
b & = (s^4-  t^4)^2(s^6 - s^4t^2 - 5s^2t^4 + t^6)^2,\\
c & = 4s^2t^4(s^2 + t^2)(3s^6 + s^4t^2 + s^2t^4 - t^6)^2,
\end{aligned}
\label{parmsol3}
\end{equation}
where $s$ and $t$ are arbitrary parameters.

Similarly, applying Choudhry's theorem to Eq.        \eqref{quarticec}, again taking $(u_1, v_1)$ as the solution $(0, s^2+t^2)$ and $(u_2, v_2)$ as the solution \eqref{Fermatsol2}, we find a solution of Eq.        \eqref{quarticec} in which 
\begin{equation}
u = 2s^3(s^6 - s^4t^2 - 5s^2t^4 + t^6)/((3s^6 + s^4t^2 + s^2t^4 - t^6)(s^2 - t^2)) , \label{RMJsol2}
\end{equation}
and this yields the following solution of our problem:
\begin{equation}
\begin{aligned}
a & =t^2(s^2 - t^2)^2(s^2 + t^2)(3s^6 + s^4t^2 + s^2t^4 - t^6)^2 ,\\
b & = 4s^4t^4(3s^6 + s^4t^2 + s^2t^4 - t^6)^2,\\
c & =s^2(s^2 - t^2)^2(s^2 + t^2)(s^6 - s^4t^2 - 5s^2t^4 + t^6)^2,
\end{aligned}
\label{parmsol4}
\end{equation}
where, as before,  $s$ and  $t$ are arbitrary parameters.

We will now show that Eq.        \eqref{quarticec} has infinitely many rational solutions. We first write,
\begin{equation}
t=ms, \quad u=sU, \quad v=s^2V,
\label{valtuv}
\end{equation}
when Eq.        \eqref{quarticec} reduces to
\begin{multline}
V^2 = U^4 -4(m^2+1)U^3 - 2(m^2 + 1)(2m^4 - 2m^2 - 3)U^2\\
 - 4(m^2 + 1)^2U + (m^2 + 1)^2. \label{quarticecUV}
\end{multline}

While  Eq.       ~\eqref{quarticecUV} is in three variables $m, U $ and $V$, and hence  represents an algebraic surface over the field $\mathbb{Q}$, we may regard it as representing  a quartic model of an elliptic curve over the function field $\mathbb{Q}(m)$.  The birational transformation defined by
\begin{equation}
\begin{aligned}
U&=(6(m^2+1)X+mY+72m^6 + 72m^4 - 72m^2  - 72)\\
& \quad \quad \times (6(X-24m^4 - 36m^2 - 12))^{-1},\\
V&=(2m^2X^3 - 36m^2(2m^2 + 1)(m^2 + 1)X^2 - m^2Y^2 \\
 & \quad \quad- 432m^3(m^2 + 1)^2Y+ 1728m^2(m^2+1)^3(8m^6 - 15m^4 - 21m^2 + 1))\\
 & \quad \quad \times (6(X-24m^4 - 36m^2 - 12))^{-2}, \\
\end{aligned}
\label{biratUV}
\end{equation}
and 
\begin{equation}
\begin{aligned}
X&=6(3U^2 - (6m^2 + 6)U + 3V - (m^2 + 1)(2m^4 - 2m^2 - 3))/m^2,\\
 Y& = 108(U^3 -(3m^2 + 3)U^2 + UV - (m^2 + 1)(2m^4 - 2m^2 - 3)U\\
 & \quad \quad - (m^2 + 1)V - (m^2 + 1)^2)/m^3,
\end{aligned}
\label{biratXY}
\end{equation}
reduces  Eq.        \eqref{quarticecUV} to the following cubic equation:
\begin{equation}
\begin{aligned}
Y^2 & = X^3 - 432(m^4 - 2m^2 - 2)(m^2 + 1)^2X \\
& \quad \quad - 1728(m^2 - 1)(m^2 + 1)^3(2m^4 - 4m^2 - 7). 
\end{aligned}
\label{ecweier}
\end{equation}

Eq.~\eqref{ecweier} also represents an algebraic surface over $\mathbb{Q}$, but as in the case of Eq.       ~\eqref{quarticecUV}, we may regard it as  the Weierstrass  model of an elliptic curve over the function field $\mathbb{Q}(m)$. Using the solution \eqref{Fermatsol1} of Eq.        \eqref{quarticec} and the relations \eqref{valtuv} and \eqref{biratXY}, we readily find that a rational point  $P$ on the elliptic curve \eqref{ecweier} is given by
\[
X=-12(m^6 - 4m^2 - 3)/m^2, \quad Y = 216(m^2 + 1)^2/m^3.
\]

When $m=4$, the curve \eqref{ecweier} reduces to
\begin{equation}
Y^2 =  X^3 - 27716256X - 56159127360, \label{ecspl}
\end{equation}
and, corresponding to the point $P$ on the curve \eqref{ecweier},  a rational point on the  curve \eqref{ecspl}   is $( -12087/4,  7803/8)$. Since this rational point on the elliptic curve \eqref{ecspl} does not have integer coordinates,  it follows from the Nagell-Lutz theorem \cite[p. 56]{Si} on elliptic curves that the point $(-12087/4,  7803/8)$ is not a point of finite order.  We can thus find infinitely many rational points on the curve \eqref{ecspl} using the group law.

Since in the special case $m=4$, the point on the curve \eqref{ecspl} corresponding to the point $P$ is not of finite order, it follows that the point $P$ on the curve \eqref{ecweier} cannot  be a point of finite order. We can thus generate infinitely many rational points on the curve \eqref{ecweier} using the group law. Each of these rational points will yield a corresponding rational point on the curve \eqref{quarticecUV} and using the relations \eqref{valtuv},  we can find infinitely many parametric solutions of the diophantine Eq.        \eqref{quarticec}. We can thus obtain infinitely many parametric solutions of our problem.

Finally, we note that all the above solutions have been obtained by assuming that both sides of Eq.        \eqref{cubeqn2} are $\neq 0$. If we take both sides of Eq.        \eqref{cubeqn2} and then solve Eq.        \eqref{cubeqn}, we do not get any new solutions of our problem.

\subsection{Parametric solutions in which all the three integers are squares}\label{threesquares}
Solutions in which all  the three integers $a, b, c$ are squares may be obtained from each of the four solutions \eqref{parmsol1}, \eqref{parmsol2}, \eqref{parmsol3} and \eqref{parmsol4} by choosing the parameters $s$ and $t$ such that $s^2+t^2$ is a square. Thus, on taking $s=2mn$ and $ t=m^2+n^2$ in the two solutions \eqref{parmsol1}, \eqref{parmsol2}, we get the following two  triads of squares such that their three elementary symmetric functions are squares:

\begin{equation}
\begin{aligned}
a  & =(m^4-n^4)^2(m^4 - 6m^2n^2 + n^4)^2, \\
b & = 4m^2n^2(m^2+n^2)^2(m^4 - 6m^2n^2 + n^4)^2,\\
c & =  64m^4n^4(m^2-n^2)^4,
\label{allsquaresol1}
\end{aligned}
\end{equation}
and
\begin{equation}
\begin{aligned}
a & = 64m^4n^4(m^2 - n^2)^2, \\
b & = 16m^2n^2(m^2 - n^2)^4, \\
c & = (m^2 + n^2)^2(m^4 - 6m^2n^2 + n^4)^2
\label{allsquaresol2}
\end{aligned}
\end{equation}
where $m$ and $n$ are arbitrary parameters.

Similarly, the two solutions \eqref{parmsol3} and \eqref{parmsol4} yield the two triads,
\begin{equation}
\begin{aligned}
a & = 64m^4n^4(m^2 - n^2)^2(m^{{12}} - 26m^{{10}}n^2 + 79m^8n^4 \\
& \quad \quad - 44m^6n^6 + 79m^4n^8 - 26m^2n^{{10}} + n^{{12}})^2, \\
b & = (m^2 + n^2)^2(m^4 - 6m^2n^2 + n^4)^2(m^{12} - 26m^{10}n^2 + 79m^8n^4\\
& \quad \quad - 44m^6n^6 + 79m^4n^8 - 26m^2n^{10} + n^{12})^2,\\
c & =  16m^2n^2(m^2 - n^2)^4(m^{12} - {10}m^{10}n^2 + 15m^8n^4\\
& \quad \quad - 204m^6n^6 + 15m^4n^8 - {10}m^2n^{10} + n^{12})^2
\label{allsquaresol3}
\end{aligned}
\end{equation}
and
\begin{equation}
\begin{aligned}
a & = (m^4 - n^4)^2(m^4 - 6m^2n^2 + n^4)^2(m^{12} - {10}m^{10}n^2 \\
 & \quad \quad + 15m^8n^4- 204m^6n^6 + 15m^4n^8 - {10}m^2n^{10} + n^{12})^2, \\
b & = 64m^4n^4(m^2 - n^2)^4(m^{12} - {10}m^{10}n^2 + 15m^8n^4\\
& \quad \quad  - 204m^6n^6 + 15m^4n^8 - {10}m^2n^{10} + n^{12})^2,\\
c & =  4m^2n^2(m^2 + n^2)^2(m^4 - 6m^2n^2 + n^4)^2(m^{12} - 26m^{10}n^2 \\
& \quad \quad + 79m^8n^4- 44m^6n^6 + 79m^4n^8 - 26m^2n^{10} + n^{12})^2
\label{allsquaresol4}
\end{aligned}
\end{equation}
where, as in the case of the two triads \eqref{allsquaresol1} and \eqref{allsquaresol2}, $m$ and $n$ are arbitrary parameters.

\subsection{Parametric solutions in which all the three integers are sums of two nonzero squares}\label{nosquares}
We will now obtain  three integers, each expressible as a   sum of two nonzero squares, such that their elementary symmetric functions are all squares. 

To obtain the three desired integers, we must choose the parameters $f, g$ and $h$  such that all the three roots of the cubic Eq.        \eqref{cubeqn} are expressible as sums of two nonzero squares. With $f$ and $g$ defined by \eqref{valfg}, one root of Eq.       \eqref{cubeqn} is $p^2+q^2$, as desired, and after removing the factor $x-p^2-q^2$, we had reduced Eq.        \eqref{cubeqn} to the quadratic Eq.        \eqref{quadxu}. Thus, we need to choose the parameters $s, t$ and $u$ such that both roots of Eq.        \eqref{quadxu} are expressible as sums of two nonzero squares.

Since the product of the two roots of Eq.        \eqref{quadxu} is $u^2(s^2+t^2)$, it follows from Lemma \ref{prelemma}, that if one of the two roots is a sum of two nonzero squares and this root is not expressible as $k^2(s^2+t^2)$ where $k$ is a nonzero rational number, then the second root of Eq.        \eqref{quadxu} will also be a sum of two nonzero squares. 

Now Eq.        \eqref{quadxu} may be considered as a quadratic equation in $u$ and it will be satisfied by a rational value of $u$ if its discriminant $4t^2\psi(s, t, x)$ with respect to $u$ is a square, where 
\begin{equation}
\psi(s, t, x) = x(s^2 - x)t^4 + 2s^4t^2x + s^2(s^4 + s^2x + x^2)x.  \label{defpsi}
\end{equation}

As we have already noted, the rational value of $u$ given by \eqref{Fermatsol1} yields two roots of Eq. \eqref{quadxu} given by \eqref{rootpairquadxu}, the first root being $ x=s^2(s^2+t^2)/t^2$. It follows that, with this value of $x$, the discriminant of Eq.       \eqref{quadxu} with respect to $u$, and hence also $\psi(s, t, x)$, must be a  square, and further, if we take 
\begin{equation}
x=s^2(r^2+s^2)/r^2, \label{valxspl}
\end{equation} 
where $r$ is an arbitrary rational parameter, then  $\psi(s, t, x)$ must become a perfect square when $t=r$.  This can also be verified by direct computation.

Thus when we take $x$ as defined by \eqref{valxspl}, then  $\psi(s, t, x)$ is  a quartic function of $t$ which becomes a square when $t=r$. Using this known value of $t$, we apply the aforementioned method of Fermat to find the following value of $t$ which also makes  $\psi(s, t, x)$  a  square:
\begin{equation}
t = r(r^6 - 9r^4s^2 - 9r^2s^4 - 3s^6)/(5r^6 + 3r^4s^2 + 3r^2s^4 + s^6).
\label{valt}
\end{equation}
Since $\psi(s, t, x)$ is now a  square, Eq.        \eqref{quadxu} has two rational solutions for $u$, one of which is given below:
\begin{equation}
u =  2s^3(5r^6 + 3r^4s^2 + 3r^2s^4 + s^6)/(r^8 + 2r^6s^2 - 12r^4s^4 - 6r^2s^6 - s^8). \label{valu}
\end{equation}
We omit writing the second rational value of $u$ that satisfies   Eq.        \eqref{quadxu} as it is cumbersome to write. We note that  the values of $x, t$ and $u$ given by \eqref{valxspl}, \eqref{valt} and \eqref{valu}, respectively, give a rational solution of Eq.        \eqref{quadxu}.

In Eq.        \eqref{quadxu}, we now take $t$ and $u$ as defined by \eqref{valt} and \eqref{valu}, respectively, and consider it as a quadratic equation in $x$. Since \eqref{valxspl} gives one rational solution of this equation, the second solution is also rational, and is given by
\begin{multline}
x = 4r^2s^4(r^{12} + 6r^{10}s^2 + 87r^8s^4 + 108r^6s^6 + 55r^4s^8 + 14r^2s^{10} + s^{12})\\
\times (r^8 + 2r^6s^2 - 12r^4s^4 - 6r^2s^6 - s^8)^{-2}. \label{valxspl2}
\end{multline}

We have thus obtained  two rational values of $x$ that are roots of Eq.        \eqref{quadxu} and  both of them are expressible as sums of two nonzero  squares. Thus, Eq.        \eqref{cubeqn} has three rational roots, and hence, as before, on appropriate scaling, we get the following three integers $a, b, c$ whose elementary symmetric functions are all squares:
\begin{equation}
\begin{aligned}
a& = r^2(r^2 + s^2)(r^8 + 2r^6s^2 - {12}r^4s^4 - 6r^2s^6 - s^8)^2\\
  & \quad \quad \times (r^{12} + 6r^{10}s^2 + 87r^8s^4 + {10}8r^6s^6 + 55r^4s^8 + 14r^2s^{10} + s^{12}),\\
b& = s^2(r^2 + s^2)(5r^6 + 3r^4s^2 + 3r^2s^4 + s^6)^2\\
  & \quad \quad \times (r^8 + 2r^6s^2 - {12}r^4s^4 - 6r^2s^6 - s^8)^2,\\
c  & =  4r^4s^4(5r^6 + 3r^4s^2 + 3r^2s^4 + s^6)^2\\
   & \quad \quad \times (r^{12} + 6r^{10}s^2 + 87r^8s^4 + {10}8r^6s^6 + 55r^4s^8 + 14r^2s^{10} + s^{12}),
\end{aligned}
\label{gensol1}
\end{equation}	
where $r$ and $s$ are arbitrary integer parameters. 

We now indicate how more such parametric solutions can be obtained. When we had mentioned the value of $u$ given by \eqref{valu}  as a rational solution of Eq.        \eqref{quadxu}, we had omitted writing the  second rational value of $u$ that satisfies   Eq.  \eqref{quadxu}.  We can use this second solution, and proceeding exactly as before, we get a second parametric solution, in terms of polynomials of degree 52,  for the three integers $a, b, c$ whose elementary symmetric functions are all squares. As this solution is  cumbersome to write, we do not give it explicitly. 

Finally, we note that more such solutions can be found by making $\psi(s, t, x)$ a  square where $x$ is defined by \eqref{valxspl}. As already noted, $\psi(s, t, x)$  is a  quartic function of $t$ which becomes a square when $t=r$ and also when $t$ is given by \eqref{valt}. Just as we proved  in  Section \ref{onesquare} that there exist infinitely many values of $u$ that make the quartic function $\phi(s, t, u)$ a square, we can now prove that  there exist infinitely many values of $t$ that make the function $\psi(s, t, x)$ a square, and these values of $t$ yield infinitely many parametric solutions of our problem in which the three integers $a, b, c$ are sums of two nonzero squares.

\subsection{Numerical examples}\label{numex}
By assigning small numerical values to the parameters in the parametric solutions obtained in Section \ref{threenumbers}, we obtained 21  numerical examples of triads of integers $ < 10^6$ such that their elementary symmetric functions are all squares. In fact, seven of these examples consist of integers $< 10000$. Table I lists these numerical examples --- in the last column we have indicated the specific parametric solution and the  values of the parameters that have yielded the numerical example. In a few cases, the triad given in Table I has been obtained after factoring out any squared factor common to the three integers obtained from the parametric  solution.

\begin{center}
\begin{tabular}{|c|c|c|c|} \hline
\multicolumn{4}{|c|}{\bf  Table I : Triads of integers $a, b, c$} \\
\multicolumn{4}{|c|}{\bf whose elementary symmetric functions are squares
} \\
\hline
$a$  &  $b$ & $c$ & Remarks \\
\hline
 180    & 45         & 64     & Solution \eqref{parmsol1} with $(s , t) = (1, 2) $ \\ 
1440    & 160        & 81     & Solution \eqref{parmsol1} with $(s  , t) = (1, 3) $ \\ 
61200    & 3825      & 1024     &      Solution \eqref{parmsol1} with $(s, t) = (1, 4) $  \\
2925    &      1300    &      5184     &      Solution \eqref{parmsol1} with $(s , t) = (2, 3) $   \\ 
93600    &      3744    &      625     &      Solution \eqref{parmsol1} with $(s, t) = (1, 5) $   \\
319725    &      51156    &      40000     &      Solution \eqref{parmsol1} with $(s , t) = (2, 5) $          \\
54400    &      19584    &      50625     &      Solution \eqref{parmsol1} with $(s  , t) = (3, 5) $          \\
83025    &      53136    &      640000     &      Solution \eqref{parmsol1} with $(s , t) = (4, 5) $          \\
80  &           320  &           225        &           Solution \eqref{parmsol2} with $(s, t) = (1, 2) $              \\          
90  &           810  &           1600  &           Solution \eqref{parmsol2} with $(s, t) = (1, 3) $  \\          
1088  &           17408  &           65025  &           Solution \eqref{parmsol2} with $(s, t) = (1, 4) $              \\          
7488  &           16848  &           4225  &           Solution \eqref{parmsol2} with $(s, t) = (2, 3) $              \\          
650  &           16250  &           97344  &           Solution \eqref{parmsol2} with $(s, t) = (1, 5) $              \\          
46400  &           290000  &           370881  &           Solution \eqref{parmsol2} with $(s, t) = (2, 5) $              \\ 
98  &           4802  &           57600  &           Solution \eqref{parmsol2} with $(s, t) = (1, 7) $              \\          
68850  &           191250  &           73984  &           Solution \eqref{parmsol2} with $(s, t) = (3, 5) $              \\ 
28880 &            81225 &            537920 &            Solution \eqref{parmsol3} with $(s , t) = (1, 2) $              \\ 
302580 &            107584 &            16245 &            Solution \eqref{parmsol4} with $(s, t) = (1, 2) $              \\ 
11025 &            19600 &            82944 &            Solution \eqref{allsquaresol1} with $(m, n) = (1, 2) $           \\ 
9216 &            5184 &            1225&            Solution \eqref{allsquaresol2} with $(m, n) = (1, 2) $           \\
136 &            72 &            153 &              Solution \eqref{gensol1} with $(r , s) = (1, 1) $           \\          
\hline
\end{tabular}
\end{center}

\section{Concluding remarks}
In this paper we reconsidered Euler's  diophantine problem of finding three integers whose elementary symmetric functions are all squares. Euler had found this to be a very difficult problem and had obtained solutions in very large integers. We first showed that each of the three integers must be a sum of two rational squares, and using this fact, we were able to find several parametric solutions that are much simpler than the solutions obtained by Euler's methods. These parametric solutions yielded a number of numerical examples of our problem in  small integers.

\noindent Ajai Choudhry, 13/4 A Clay Square, Lucknow - 226001, India.

\noindent E-mail address: ajaic203@yahoo.com
\end{document}